\documentclass[11pt]{amsart}

\usepackage{amsmath}
\usepackage{amsthm}
\usepackage{amsfonts}
\usepackage{xcolor}
\usepackage[english]{babel}
\usepackage[margin=1.09in]{geometry}
\usepackage[colorlinks=true,
        linkcolor=blue]{hyperref}

        \definecolor{pink}{rgb}{1,0,1}

\usepackage[latin9]{inputenc}
\usepackage{amsmath}
\usepackage{amsfonts}
\usepackage{amssymb}
\usepackage{amsthm}

\newtheorem{theo}{Theorem}[section]
\newtheorem{prop}[theo]{Proposition}
\newtheorem{coro}[theo]{Corollary}
\newtheorem{lemm}[theo]{Lemma}

\theoremstyle{definition}

\theoremstyle{remark}
\newtheorem{rema}[theo]{Remark}

\newcommand{\nwc}{\newcommand}
\nwc{\eps}{\epsilon}
\nwc{\vareps}{\varepsilon}
\nwc{\Oph}{\operatorname{Op}_\hbar}
\nwc{\la}{\langle}
\nwc{\ra}{\rangle}

\nwc{\mf}{\mathbf} 
\nwc{\blds}{\boldsymbol} 
\nwc{\ml}{\mathcal} 

\nwc{\defeq}{\stackrel{\rm{def}}{=}}

\nwc{\cE}{\ml{E}}
\nwc{\cN}{\ml{N}}
\nwc{\cO}{\ml{O}}
\nwc{\cP}{\ml{P}}
\nwc{\cU}{\ml{U}}
\nwc{\cV}{\ml{V}}
\nwc{\cW}{\ml{W}}
\nwc{\tU}{\widetilde{U}}
\nwc{\IN}{\mathbb{N}}
\nwc{\IR}{\mathbb{R}}
\nwc{\IZ}{\mathbb{Z}}
\nwc{\IC}{\mathbb{C}}
\nwc{\IT}{\mathbb{T}}
\nwc{\tP}{\widetilde{P}}
\nwc{\tPi}{\widetilde{\Pi}}
\nwc{\tV}{\widetilde{V}}
\nwc{\supp}{\operatorname{supp}}
\nwc{\rest}{\restriction}

\renewcommand{\phi}{\varphi}

\title[Equidistribution of toral eigenfunctions]
{Equidistribution of toral eigenfunctions along hypersurfaces}

\author{Hamid Hezari }
\address{Department of Mathematics, UC Irvine, Irvine, CA 92617, USA} \email{hezari@math.uci.edu}

\author[Gabriel Rivi\`ere]{Gabriel Rivi\`ere}
\address{Laboratoire Paul Painlev\'e (U.M.R. CNRS 8524), U.F.R. de Math\'ematiques, Universit\'e Lille 1, 59655 Villeneuve d'Ascq Cedex, France} \email{gabriel.riviere@math.univ-lille1.fr}

\date{\today}

\begin{document}

\begin{abstract} We prove a new quantum variance estimate for toral eigenfunctions. As an application, we show that, given 
any orthonormal basis of toral eigenfunctions 
and any smooth embedded hypersurface with nonvanishing principal curvatures, 
there exists a density one subsequence of eigenfunctions that 
equidistribute along the hypersurface. This is an analogue of the \emph{Quantum Ergodic Restriction} theorems in the case of the flat torus, which in particular verifies 
the Bourgain-Rudnick's conjecture on $L^2$-restriction estimates 
for a density one subsequence of eigenfunctions 
in any dimension. 
Using our quantum variance estimates, we also obtain equidistribution of eigenfunctions against measures whose supports have Fourier dimension larger than $d-2$. In the end,  we also describe a few quantitative results specific to dimension $2$.

\end{abstract}

\maketitle
\section{Introduction}

Let $(M,g)$ be a smooth ($\ml{C}^{\infty}$), compact, oriented, Riemannian manifold without boundary and of dimension $d\geq 2$. Consider the following eigenvalue problem
\begin{equation}\label{e:eigenvalue}
 -\Delta_g\psi_{\lambda}=\lambda^2\psi_{\lambda},\quad\|\psi_{\lambda}\|_{L^2}=1.
\end{equation}
Let now $\Sigma$ be a smooth, compact, oriented and embedded submanifold of $M$ of codimension\footnote{We will refer to such submanifolds as hypersurfaces.} $1$. 
Then, it is a natural question to estimate the $L^2$-norm of the restrictions of eigenfunction to $\Sigma$ with respect to the induced hypersurface measure $d\sigma$. 
It was proved by Burq, G\'erard and Tzvetkov~\cite{BGT07} 
that, for any solution $\psi_{\lambda}$ of~\eqref{e:eigenvalue}, one has:
\begin{equation}\label{e:BGT}\|\psi_{\lambda}\|_{L^2(\Sigma)}\leq C_{\Sigma}\lambda^{\frac{1}{4}},\end{equation}
where $C_{\Sigma}>0$ depends only on $(M,g)$ and $\Sigma$. They also proved that this bound is sharp in general. In addition, they proved that for $d=2$ and if $\Sigma$ has non vanishing geodesic curvature, the bound can be improved to
$$\|\psi_{\lambda}\|_{L^2(\Sigma)}\leq C_{\Sigma}\lambda^{\frac{1}{6}},$$
which is again sharp in general. Though these estimates are sharp, it must be noticed that sequences for which the upper bound is sharp are in some sense sparse. 
Indeed, by the H\"ormander-Weyl's law~\cite{Ho68}, given any orthonormal basis (ONB) $(\psi_j)_{j\in\mathbb{N}}$ of Laplace eigenfunctions, one has uniformly for $x \in M$,
$$\frac{1}{N(\lambda)}\sum_{j:\lambda_j\leq \lambda}|\psi_{j}(x)|^2=\frac{1}{\text{Vol}_g(M)}+\ml{O}(\lambda^{-1}),$$
where $N(\lambda):=|\{j:\lambda_j\leq\lambda\}|$ is the spectral counting function. After 
integrating this asymptotics against the hypersurface measure $\sigma$, this implies that, for any given function $R(\lambda) \to \infty$ and any 
ONB of eigenfunctions, there is a full density\footnote{Recall that $S\subset \mathbb{N}$ has full density (or density $1$)
if $|\{j\in S:j\leq N\}|/N\rightarrow 1$ as $N\rightarrow+\infty.$} subsequence $(\psi_j)_{j\in S}$ of Laplace eigenfunctions such that
$$\forall j\in S,\quad\|\psi_{\lambda_j}\|_{L^2(\Sigma)}\leq R(\lambda_j).$$

Another way to look for improvements is to make some restriction on the geometry of the manifold. For instance, the authors of \cite{BGT07} also 
noted that, in the case of the $2$-dimensional flat torus $\mathbb{T}^2:=\IR^2/(2\pi\IZ)^2$, their bound~\eqref{e:BGT} can be drastically 
improved as the $L^{\infty}$-norms are controlled by a term of order $\lambda^{\delta}$ for every $\delta>0$ without any extraction argument. 
Improving this bound in the case of the flat torus was recently pursued by Bourgain and Rudnick~\cite{BoRu12} who obtained uniform lower and upper bounds on $\|\psi_{\lambda}\|_{L^2(\Sigma)}$ for $d=2$ or $3$ 
when $\Sigma$ is a real analytic hypersurface with nonvanishing curvatures -- see below for a more precise statement. 
They also conjectured that these quantities should be uniformly bounded in any dimension whenever the hypersurface 
is real analytic~\cite[Conj.~1.9]{BoRu12}. As an application of our analysis, we verify this conjecture \emph{for a density $1$ subsequence of 
eigenfunctions provided some nonvanishing curvature properties are satisfied by $\Sigma$ and only smoothness} -- see Corollary~\ref{c:density1-hypersurface} below. More precisely,  
we show that for any given function $R(\lambda) \to \infty$ and any ONB of eigenfunctions, there is a full density subsequence of eigenfunctions $(\psi_j)_{j\in S}$, 
along which one has
$$ \int_\Sigma |\psi_{j}|^2 d \sigma = \text{Vol}(\Sigma) + \ml{O} \left ( \lambda_j^{-\frac12} R(\lambda_j)  \log \lambda_j \right ), $$
where $ \text{Vol}(\Sigma)$ is the hypersurface volume of $\Sigma$.

Related to these restriction estimates are the Quantum Ergodicity Restriction (QER) theorems of Toth--Zelditch~\cite{ToZe12, ToZe13} and Dyatlov--Zworski~\cite{DyZw12}. Recall 
that the Quantum Ergodicity (QE) theorem states that almost all the eigenfunctions of a given orthonormal basis of Laplace eigenfunctions become 
equidistributed in the unit cotangent bundle if the Liouville measure is ergodic~\cite{Sh74, Ze87, CdV85}. It is 
natural to ask if this equidistribution remains true along hypersurfaces and this question was recently answered by the 
above mentioned works on QER theorems. In particular, it implies that, for such subsequences, the $L^2$-restriction estimates can be improved as 
$\|\psi_{\lambda}\|_{L^2(\Sigma)}^2$ indeed converges to the volume of $\Sigma$. The QE theorem can be extended to other dynamical 
frameworks, and the analogue of this result was obtained by Marklof and Rudnick in the case of rational polygons including the case of the flat torus 
$\mathbb{T}^d:=\IR^d/(2\pi\IZ)^d$~\cite{MaRu12}. The main difference in this case is that most 
eigenfunctions become equidistributed in configuration space but not necessarily in phase space. Quantitative versions 
of this result were recently derived by techniques of different nature. Our proof in~\cite{HeRi17} was close to the original proof of the QE 
theorem while the proof of Lester and Rudnick in~\cite{LeRu17} made use of tools of arithmetic nature. The main purpose of the present article is to 
combine the methods from these two articles and obtain improvements on some of the results from these references and also 
to deduce analogues of the QER theorem in the case of the flat torus. In particular, we will 
prove that almost all toral eigenfunctions equidistribute along hypersurfaces with nonvanishing principal curvatures.

\section{Statement of the main results}

Let us fix some notations. Throughout the paper, we denote by $\mathbb{T}^d:=\IR^d/(2\pi\IZ)^d$ the rational torus with the standard metric, 
by $dx$ the \emph{normalized volume measure} induced by the standard metric, and by $(\psi_j)_{j\geq 1}$ an orthonormal basis 
of Laplace eigenfunctions, i.e.
$$-\Delta\psi_j=\lambda_j^2\psi_j,$$
with $0=\lambda_1<\lambda_2\leq\lambda_3\leq\ldots\lambda_n\ldots$. The set of eigenvalues (counted with multiplicities) 
can be indexed by lattice points of $\IZ^d$, i.e.
$\left\{\lambda_j^2:j\geq 1\right\}=\left\{\|n\|^2:n\in\IZ^d\right\}.$ We shall denote by $N(\lambda)$ the spectral 
counting function over \emph{long intervals}, 
$$N(\lambda):=\left|\left\{j\geq 1:\lambda_j\leq\lambda\right\}\right|=\left|\left\{n\in\IZ^d:\|n\|\leq\lambda\right\}\right|.$$

\subsection{Equidistribution along hypersurfaces}

Our main result is the following:
\begin{theo}\label{t:maintheo-hypersurface} Let $\Sigma\subset\IT^d$ be a smooth compact embedded and oriented submanifold of 
dimension $d-1$ which has no boundary. Suppose that, for every $x$ in $\Sigma$, all the principal curvatures of $\Sigma$ at $x$ are different from $0$. 
 Then, for any $a$ in $\ml{C}^{\infty}(\IT^d)$, there exists a constant $C_a>0$ such that, for any orthonormal basis $(\psi_j)_{j\geq 1}$ of 
 eigenfunctions of $\Delta$, one has
 $$\frac{1}{N(\lambda)}\sum_{j:\lambda_j\leq \lambda}\left|\int_{\Sigma}a|\psi_j|^2d\sigma-\int_{\Sigma}ad\sigma\right|^2\leq 
 C_a\frac{(\log\lambda)^2}{\lambda},$$
  where $\sigma$ is the induced hypersurface measure on $\Sigma$.
\end{theo}
In other words, this theorem states that, on average, eigenfunctions are equidistributed along $\Sigma$ with a precise rate of convergence. 
In fact the proof will show that the error is of order $\frac{\log\lambda}{\lambda}$ for $d\geq 3$ -- see Remark~\ref{r:log-term}. As we shall see, this 
theorem is a corollary of Littman's theorem~\cite{Li63} (see Remark~\ref{r:littman} below) and of our Theorem~\ref{t:QE-measure} which is slightly more general. 
An extraction argument~\cite{Ze87, CdV85} allows to show that:

\begin{coro}[Equidistribution along hypersurfaces]\label{c:density1-hypersurface} Suppose that the assumptions of Theorem~\ref{t:maintheo-hypersurface} are satisfied. Then,
for any orthonormal basis $(\psi_j)_{j\geq 1}$ of eigenfunctions of $\Delta$, there exists a density $1$ subset $S$ of $\IN$ such that
 \begin{equation} \label{equi} \forall a\in\ml{C}^0(\IT^d),\quad\lim_{j\rightarrow+\infty, j\in S}\int_{\Sigma}a|\psi_j|^2d\sigma=\int_{\Sigma}ad\sigma, \end{equation}
 where $\sigma$ is the induced hypersurface measure on $\Sigma$.  Moreover, for any function $R(\lambda) \to \infty$, there is a density $1$ subset $S$ along which, in addition to \eqref{equi}, one also has
$$ \left|\int_\Sigma |\psi_{j}|^2 d \sigma - \operatorname{Vol}(\Sigma) \right|\leq  \frac{R(\lambda_j)  \log \lambda_j}{\lambda_j^{\frac12}} , $$
where $ \operatorname{Vol}(\Sigma)$ is the hypersurface volume of $\Sigma$. 
\end{coro}

This result is the analogue in the case of the flat torus of the Quantum Ergodicity Restriction theorems from~\cite{ToZe12, ToZe13, DyZw12}. The results from these 
references are stronger in the sense that equidistribution also holds in phase space but we emphasize they do not provide any rate of convergence of the variance. It is plausible 
that, combined with the arguments from~\cite{Ze94, Sch06}, one would get a logarithmic decay rate when $(M,g)$ is negatively curved. Recall also that 
these QER statements are valid under the assumptions that the geodesic flow is ergodic for the Liouville measure (which is not the case here) and that the hypersurface verifies a certain 
asymmetry condition on the geodesics passing through $\Sigma$. While these conditions are of dynamical nature, we emphasize that ours are purely geometrical.

A direct consequence of this corollary is that, along a density one subsequence, the $L^2$-restriction estimates of~\cite{BGT07} can be improved. Recall, that 
for a smooth closed embedded and oriented curve $\Sigma\subset\IT^2$ with nonvanishing curvature, 
Bourgain and Rudnick proved~\cite{BoRu12} that there exist $0<c\leq C$ such that, for every $\psi_{\lambda}$ satisfying
$\Delta\psi_{\lambda}=-\lambda^2\psi_{\lambda},$ one has
\begin{equation}\label{e:bourgain-rudnick}c\|\psi_{\lambda}\|_{L^2(\IT^2)}\leq\|\psi_{\lambda}\|_{L^2(\Sigma)}\leq C\|\psi_{\lambda}\|_{L^2(\IT^2)}.\end{equation}
They also proved (in a more involved proof) that this remains true in dimension $3$ for \emph{real analytic}\footnote{In the case $d=2$, the short proof 
given in~\cite[p.880--881]{BoRu12} does not require analyticity.} hypersurfaces with nonzero curvature. In the higher 
dimensional case, it seems that the question remains open. It is in fact conjectured in~\cite[Conj.~1.9]{BoRu12} that, under appropriate 
assumptions on $\Sigma$, these two bounds should hold in any dimension. In particular, the above corollary is consistent with that conjecture as it 
shows it is true in any dimension for a \emph{density 1} subsequence of eigenfunctions provided that the principal curvatures of $\Sigma$ do not vanish. Along this 
typical sequence, the result is even better than expected as it shows that $\|\psi_{j}\|_{L^2(\Sigma)}$ converges in any dimension. 
Applying~\cite[Th.~0.2 and 0.3]{ToZe17} would give upper bounds on the size of the nodal sets restricted to $\Sigma$ but this would not be better than the bounds 
from~\cite{BoRu12} which are valid without extracting subsequences.

\subsection{Rate of decay of the quantum variance}

Let us now come back to the more classical framework of the quantum ergodicity theorems where we consider equidistribution on $\mathbb T^d$
and not along hypersurfaces. In this case, 
we can consider averages over shorter spectral intervals and show:
\begin{theo}\label{t:maintheo-short} There exists a universal constant $C_d>0$ such that, 
for any $a$ in $L^2(\IT^d)$, for any orthonormal basis $(\psi_j)_{j\geq 1}$ of 
 eigenfunctions of $\Delta$, and for any $\lambda>0$, one has
$$\frac{1}{\left|\left\{j:\lambda-1\leq\lambda_j\leq\lambda\right\}\right|}
\sum_{j:\lambda-1\leq \lambda_j\leq \lambda}\left|\int_{\IT^d}a|\psi_j|^2dx-\int_{\IT^d}adx\right|^2
\leq \frac{C_d}{\lambda}\|a\|_{L^2(\IT^d)}^2.$$
\end{theo}
We cannot expect to have a better rate of decay than $\lambda^{-1}$ (at least in dimension $2$). 
Indeed, the Weyl's law~\cite{DuGu75} tells us that, in dimension $d=2$,
$$\left|\left\{j:\lambda-1\leq\lambda_j\leq\lambda\right\}\right|\sim \lambda.$$
If we had a better rate, then it would mean that every sequence of eigenfunctions equidistribute on $\IT^2$ as $\lambda\rightarrow +\infty$ and this 
is not the case -- see for instance~\cite[Sect.~3]{LeRu17}. This result should be compared with our earlier results in~\cite[Th.~1.2]{HeRi17} 
where we showed
\begin{equation}\label{e:HR-variance}\frac{1}{\left|\left\{j:\lambda-1\leq\lambda_j\leq\lambda\right\}\right|}
\sum_{j:\lambda-1\leq \lambda_j\leq \lambda}\left|\int_{\IT^d}a|\psi_j|^2dx-\int_{\IT^d}adx\right|^2\leq C_a\lambda^{-\frac{2}{3}},\end{equation}
for some constant $C_a>0$ depending on a certain number of derivatives of $a$. Our proof was modeled on the classical proof of 
the quantum ergodicity theorem~\cite{Sh74, Ze87, CdV85}. In~\cite{LeRu17}, this result was improved by Lester and Rudnick who considered 
the moments of order $1$. They obtained the following\footnote{The proof in~\cite[Prop.~2.4]{LeRu17} 
is given for long spectral intervals but it can be adapted to shorter intervals by using~\eqref{e:LR-short} below.}:
\begin{equation}\label{e:LR-variance1}\frac{1}{\left|\left\{j:\lambda-1\leq\lambda_j\leq\lambda\right\}\right|}
\sum_{j:\lambda-1\leq \lambda_j\leq \lambda}\left|\int_{\IT^d}a|\psi_j|^2dx-\int_{\IT^d}adx\right|\leq \frac{C_a}{\lambda},\end{equation}
for some constant $C_a>0$ depending again on a certain number of derivatives of $a$. Observe now that 
$$ \left|\int_{\IT^d}a|\psi_j|^2dx-\int_{\IT^d}a dx\right|^2 \leq 2 \| a \|_{\ml{C}^0} \left|\int_{\IT^d}a|\psi_j|^2dx-\int_{\IT^d}adx\right|,$$ 
which yields the following improvement of~\eqref{e:HR-variance}:
\begin{equation}\label{e:LR-variance2}\frac{1}{\left|\left\{j:\lambda-1\leq\lambda_j\leq\lambda\right\}\right|}
\sum_{j:\lambda-1\leq \lambda_j\leq \lambda}\left|\int_{\IT^d}a|\psi_j|^2dx-\int_{\IT^d}adx\right|^2\leq \frac{C_a}{\lambda}.\end{equation}
This is a priori the best estimate on the variance 
we can expect (in terms of $\lambda$) from this kind of arguments. However, we will show that the combination of the semiclassical methods from~\cite{HeRi17} with the arithmetic methods 
from~\cite{LeRu17} allows to obtain improvements on the involved constant $C_a$ for the moments of order $2$ as stated in Theorem~\ref{t:maintheo-short}. 
In dimension $2$, this type of upper bound on the variance is consistent with the 
ones appearing in the physics literature~\cite{FePe86, EFK95} -- see also~\cite{LuSa95} in the case of an arithmetic surface,~\cite{Ze94, Sch06} 
for general negatively curved surfaces, and~\cite{KuRu05} for quantum cat-maps. Finally, we observe that Theorem~\ref{t:maintheo-short} does not seem to imply ~\eqref{e:LR-variance1}.


\subsection{Recovering Zygmund's estimate} In dimension $2$, our Proposition~\ref{p:moments} (with $c(\lambda) =\lambda$) will yield that for each $\psi_j$,
$$ \left|\int_{\IT^d}a|\psi_j|^2dx-\int_{\IT^d}a dx\right|^{2} \leq \sum_{1\leq\|n\|\leq 2\lambda}|\hat{a}_n|^2\left|\left\{k:\|k\|=\|k+n\|=\lambda\right\}\right|.$$ But since the number of lattice points lying on two circles of radius 
 $\lambda$ centered at $0$ and $n\neq 0$ is at most $2$, this implies that  
  $$\left|\int_{\IT^2}a|\psi_j|^2dx-\int_{\IT^2}adx\right|^2\leq 2\int_{\IT^2}\left|a(x)-\int_{\IT^2}a(y)dy\right|^2dx.$$
By plugging $a= |\psi_j|^2$ into this we obtain
$$ \left ( \int_{\IT^2} | \psi_j|^4 dx -1 \right )^2 \leq 2 \int \left ( |\psi_j|^2 -1 \right )^2 dx,$$
which implies 
$$\forall j\geq 1,\quad  \|\psi_j\|_{L^4}\leq 3^\frac14.$$
This recovers Zygmund's classical result on  uniform boundedness of $L^4$ norms -- Zygmund's upper bound was $5^\frac14$~\cite{Zy74}. In higher dimensions, the same argument would  only give an upper bound 
of order $\lambda_j^{\frac{d-2}{4}}$ which remains far from~\cite{Bo93, BoDe14}.

\subsection{Main result on the quantum variance} Coming back to the case of long spectral intervals, we have the following crucial estimate on the quantum variance:
\begin{theo}\label{t:maintheo} There exists a universal constant $C_d>0$ such that, 
for any $a$ in $L^2(\IT^d)$, for any orthonormal basis $(\psi_j)_{j\geq 1}$ of 
 eigenfunctions of $\Delta$, and for any $\lambda>0$, one has
$$\frac{1}{N(\lambda)}\sum_{j:\lambda_j\leq \lambda}\left|\int_{\IT^d}a|\psi_j|^2dx-\int_{\IT^d}adx\right|^2
\leq \frac{C_d}{\lambda}\sum_{n:1\leq \|n\|\leq 2\lambda}\frac{|\int_{\IT^d}a(x) e^{-inx} dx|^2}{\|\hat{n}\|},$$
where $\hat{n}$ is a primitive lattice point generating $n$.
\end{theo}
 
Theorem~\ref{t:maintheo-hypersurface} is in fact a consequence of the above theorem combined with a proper application of Littman's Theorem~\cite{Li63}. 
In~\cite[Prop.~2.4]{LeRu17}, Lester and Rudnick obtained an analogous statement on moments of order $1$:
$$\frac{1}{N(\lambda)}\sum_{j:\lambda_j\leq \lambda}\left|\int_{\IT^d}a|\psi_j|^2dx-\int_{\IT^d}adx\right|\leq \frac{C_d}{\lambda}
\sum_{n:1\leq \|n\|\leq 2\lambda}\frac{\left|\int_{\IT^d}a(x) e^{-inx} dx\right|}{\|\hat{n}\|}.$$
Yet, as we shall briefly explain it in Remark~\ref{r:LR-hypersurface}, this bound on moments of order $1$ is not sufficient to prove Theorem~\ref{t:maintheo-hypersurface} except in dimension $2$. Here, we will crucially use the fact that our upper 
bound involves factors of type $|\int_{\IT^d}a(x) e^{-inx} dx|^2$. 

We also note that it is natural to study a generalized framework when the hypersurface measure 
is replaced by a general measure $\mu$ carried possibly on a complicated set.\footnote{The recent work of Eswarathasan and Pramanik~\cite{EP17} extends the results of~\cite{BGT07} when the hypersurface measure is replaced by some measure $\mu$ carried by a random fractal set.} In fact our Theorem~\ref{t:maintheo} can also be applied to deal with this situation and 
we obtain the following generalization of Theorem~\ref{t:maintheo-hypersurface}:
\begin{theo}\label{t:QE-measure} Let $\mu$ be a probability measure on $\IT^d$ such that there exists $\alpha\geq 0$ and $C>0$ satisfying
$$\forall n\in\IZ^d-\{0\},\quad|\hat{\mu}(n)|^2\leq C\|n\|^{-\alpha}.$$
 Then, for any $a$ in $\ml{C}^{\infty}(\IT^d)$, there exists a constant $C_a>0$ such that, for any orthonormal basis $(\psi_j)_{j\geq 1}$ of 
 eigenfunctions of $\Delta$, one has
 $$\frac{1}{N(\lambda)}\sum_{j:\lambda_j\leq \lambda}\left|\int_{\IT^d}a|\psi_j|^2d\mu-\int_{\IT^d}ad\mu\right|^2\leq 
 C_a(\log\lambda)^2\lambda^{\max(d-2-\alpha,-1)}.$$
\end{theo}
From Remark~\ref{r:log-term} below, the proof will in fact 
yield the better estimate $\lambda^{-1}$ for $\alpha>d-1$ which is consistent with Theorem~\ref{t:maintheo-short}. 
\begin{rema}\label{r:littman}
Written in this form, Theorem~\ref{t:maintheo-hypersurface} becomes a direct corollary of Littman's Theorem~\cite{Li63} -- see also~\cite{Hl50} for earlier results 
in the case of convex bodies. In fact, recall that this result states that, under the geometric assumptions of Theorem~\ref{t:maintheo-hypersurface}, the hypersurface measure 
$\mu=\sigma/\text{Vol}(\Sigma)$ satisfies the assumptions of Theorem~\ref{t:QE-measure} with $\alpha=d-1$. 
\end{rema}
Again, one gets
\begin{coro}\label{c:QE-measure} Let $\mu$ be a probability measure on $\IT^d$ such that there exists $\alpha>d-2$ and $C>0$ satisfying
$$\forall n\in\IZ^d-\{0\},\quad|\hat{\mu}(n)|^2\leq C\|n\|^{-\alpha}.$$
 Then, for any orthonormal basis $(\psi_j)_{j\geq 1}$ of eigenfunctions of $\Delta$, there exists a density $1$ subset $S$ of $\IN$ such that
 $$\forall a\in\ml{C}^0(\IT^d),\quad\lim_{j\rightarrow+\infty, j\in S}\int_{\IT^d}a|\psi_j|^2d\mu=\int_{\IT^d}ad\mu.$$
\end{coro}

Our assumption on the measure $\mu$ is related to the notion of Fourier dimension and Salem sets in harmonic analysis. 
Recall that, for a given compact set $K$ in $\IR^d$, the Fourier dimension of a set is defined by~\cite[p.~168]{M95}
$$\text{dim}_F(K):=\sup\left\{s\in[0,d]:\exists\ \mu\neq 0\in\ml{P}(K)\ \text{s.t.}\ \sup_{\xi\neq 0}|\hat{\mu}(\xi)|^2\|\xi\|^s<+\infty\right\},$$
where $\ml{P}(K)$ is the set of Radon measures supported in $K$. Hence, in this terminology, 
our assumptions on $\mu$ means that the support of $\mu$ has Fourier dimension $>d-2$. 

\begin{rema}\label{r:salem}
One always has $\text{dim}_F(K)\leq \text{dim}_H(K)$ where 
$\text{dim}_H(K)$ is the Hausdorff dimension of the set $K$, and the inequality is strict in general. Sets for which one has equality are referred as 
\emph{Salem sets} in the literature, and Littman's Theorem provides a class of Salem sets of dimension $d-1$. We refer to~\cite{Ha17} for explicit construction of 
$\IZ^2$-periodic Salem sets of any dimension $0<\alpha<2$ and for more references to the literature on this topic.
\end{rema}




\section{Proof of the variance estimates}

In this section, we give the proof of the variance estimates from the introduction. The main point compared with the ``semiclassical'' 
approach of~\cite{HeRi17} is that we try to optimize our argument by implementing some of the lattice points properties proved in \cite{Sc} and used 
in~\cite{LeRu17}. Let $a$ be a smooth function on $\IT^d$. We write its Fourier decomposition:
$$a=\sum_{p\in\IZ^d}\hat{a}_p\mathbf{e}_p,$$
where $\mathbf{e}_p(x)=e^{i p.x}$. Suppose for the sake of simplicity that $a$ is real valued.

\subsection{Estimating the variance}\label{ss:mainargument}

We introduce the following sum:
$$S_2(a,\lambda):=\sum_{j:\lambda_j= \lambda}\left|\int_{\IT^d}a|\psi_j|^2dx-\int_{\IT^d}adx\right|^{2}.$$
If we set\footnote{Observe that this is still real-valued.} $a_{\lambda}=\sum_{n:\|n\|\leq 2\lambda}\hat{a}_n\mathbf{e}_n$, then we can verify that
$$S_2(a,\lambda)=\sum_{j:\lambda_j= \lambda}\left|\int_{\IT^d}a_{\lambda}|\psi_j|^2dx-\int_{\IT^d}adx\right|^{2}.$$
We will now proceed as in the proof of~\cite{HeRi17} and first use the fact that 
$$\forall t\in\IR,\ e^{it\Delta}\psi_j=e^{it\lambda_j^2}\psi_j.$$
This implies that, for every $T>0$, one has
$$\int_{\IT^d}a_{\lambda}|\psi_j|^2dx=\left\la\psi_j,\left(\frac{1}{T}\int_0^Te^{-it\Delta} a_{\lambda}e^{it\Delta}dt\right) \psi_j\right\ra.$$
In order to alleviate notations, let us introduce the self-adjoint operator:
$$A(T,\lambda):=\frac{1}{T}\int_0^Te^{-it\Delta} a_{\lambda}e^{it\Delta}dt-\int_{\IT^d}adx.$$
With these conventions, one has
$$S_2(a,\lambda)=\sum_{j:\lambda_j= \lambda}\left|\la\psi_j,A(T,\lambda)\psi_j\ra\right|^{2}.$$
Apply now the Cauchy-Schwarz inequality which yields
$$S_2(a,\lambda)\leq\sum_{j:\lambda_j= \lambda}\left\|A(T,\lambda)\psi_j\right\|^{2}.$$
Using the fact that the trace of the operator $A(T,\lambda)^2$ is independent of the choice of orthonormal basis, this inequality can be rewritten as
$$S_2(a,\lambda)\leq\sum_{k:\|k\|= \lambda}\left\|A(T,\lambda)\mathbf{e}_k\right\|^{2}.$$
We now expand
\begin{eqnarray*}A(T,\lambda)\mathbf{e}_k & = &\frac{1}{T}\int_{0}^Te^{-it\Delta}\left(\sum_{1\leq\|n\|\leq 2\lambda}\hat{a}_n\mathbf{e}_n\right)e^{it\Delta}\mathbf{e}_kdt\\
 & = &\sum_{1\leq\|n\|\leq 2\lambda}\hat{a}_n\left(\frac{1}{T}\int_0^Te^{it(\|k\|^2-\|k+n\|^2)}dt\right)\mathbf{e}_{k+n}.\end{eqnarray*}
Up to this point, this is exactly the proof from~\cite{HeRi17}, but from this point on we proceed differently. First, we observe that, up to now every estimate and every identity
is valid \emph{for any} $T>0$ \emph{and any} $\lambda>0$. By letting $T$ tending to infinity first, we deduce that
 $$S_2(a,\lambda)\leq\sum_{k:\|k\|= \lambda}\left\|\sum_{1\leq\|n\|\leq 2\lambda: \|k\|^2=\|k+n\|^2}\hat{a}_{n}\mathbf{e}_{k+n} \right\|^{2},$$
 By application of Plancherel's formula, the upper bound reduces to
 $$ S_2(a,\lambda)\leq\sum_{k:\|k\|= \lambda}\sum_{\substack{1\leq\|n\|\leq 2\lambda \\ \|k\|^2=\|k+n\|^2}}|\hat{a}_n|^2,$$
 which, after changing the order of summation, yields
 \begin{equation}\label{e:moment-p=1}
  S_2(a,\lambda)\leq\sum_{1\leq\|n\|\leq 2\lambda}|\hat{a}_n|^2\left|\left\{k:\|k\|=\|k+n\|=\lambda\right\}\right|.
 \end{equation}
\begin{rema} At this stage of the proof, we would like to emphasize an important point of our argument. In the standard proof of the Quantum Ergodicity 
Theorem~\cite{Sh74, Ze87, CdV85} and also in the case of its analogue on $\IT^d$~\cite{MaRu12}, one always start by letting $\lambda$ 
tends to $+\infty$ and then $T$ to $+\infty$. Here, because of the specific structure of the torus, we can first take the large time limit and 
then the large eigenvalue limit. This observation is responsible for the improvements we get compared with the proof from~\cite{HeRi17}. 
\end{rema}
Our above estimate for $S_2(a,\lambda)$  is valid for any eigenvalue $\lambda$. Therefore, if we sum over a range of eigenvalues
$[c(\lambda),\lambda]$ (with $0\leq c(\lambda)\leq\lambda$), we get
\begin{prop}\label{p:moments} Let $0\leq c(\lambda)\leq\lambda$. For every $a$ in $L^2(\IT^d)$ and for every orthonormal basis $(\psi_j)_{j\geq 1}$ of eigenfunctions of $\Delta$, one has
\begin{eqnarray*}
 V_2(a,c(\lambda),\lambda)&:= &\frac{1}{|\{j:c(\lambda)\leq\lambda_j\leq \lambda\}|}\sum_{j:c(\lambda)\leq\lambda_j\leq \lambda}\left|\int_{\IT^d}a|\psi_j|^2dx-\int_{\IT^d}adx\right|^2\\
  &\leq &\sum_{1\leq\|n\|\leq 2\lambda}|\hat{a}_n|^2\frac{\left|\left\{k:\|k\|=\|k+n\|\in[c(\lambda),\lambda]\right\}\right|}{\left|\left\{k:\|k\|\in[c(\lambda),\lambda]\right\}\right|}.
\end{eqnarray*}
\end{prop}

\subsection{Counting lattice points}

Hence, we have reduced our question to a problem of counting lattice points which is close to the one appearing in~\cite{LeRu17} -- see also~\cite{Sc}. 
Precisely, according to~\cite[Lemma~2.3]{LeRu17}, we know that, when $c(\lambda)=0$, one has, for every $n$ in $\IZ^d-\{0\}$,
\begin{equation}\label{e:LR-long}|\{k:0\leq\|k\|\leq \lambda\ \text{and}\ \|k\|^2=\|n+k\|^2\}|\leq c_d\frac{\lambda^{d-1}}{\|\hat{n}\|},\end{equation}
for some constant $c_d>0$ depending only on $d$. Recall that $\hat{n}$ is a primitive lattice point proportional to $n$. This upper bound on 
the number of lattice points comes from the fact that we have to count the number of lattice points which are 
orthogonal to $p$ because $\|k\|^2=\|n+k\|^2$ is equivalent to the fact that $\la n,n-2k\ra=0$ -- see~\cite{LeRu17} for details. 
Together with the first part of Proposition~\ref{p:moments}, this upper bound implies Theorem~\ref{t:maintheo}.

Note that, if we consider shorter spectral intervals, the argument from~\cite[Lemma~2.3]{LeRu17} still allows to conclude that, for every $n$ in $\IZ^d-\{0\}$,
\begin{equation}\label{e:LR-short}|\{k:\lambda-1\leq\|k\|\leq \lambda\ \text{and}\ \|k\|^2=\|n+k\|^2\}|\leq c_d'\lambda^{d-2},\end{equation}
for some constant $c_d'>0$ depending only on $d$. Combined with Proposition~\ref{p:moments} and Weyl's law on the torus, this already 
implies Theorem~\ref{t:maintheo-short} from the introduction.

 \subsection{Quantum variance on individual eigenspaces} \label{r:dim2} In this section, we record some remarks when the variance is considered on an eigenspace  i.e. $c(\lambda)=\lambda$. By combining Proposition~\ref{p:moments} with~\eqref{e:LR-short}, 
we get
 \begin{equation}\label{e:variance-individual}\frac{1}{|\{j:\lambda_j= \lambda\}|}\sum_{j:\lambda_j= \lambda}\left|\int_{\IT^d}a|\psi_j|^2dx-\int_{\IT^d}adx\right|^2\leq \|a\|_{L^2(\IT^d)}^2
 \min\left\{1,c_d'\frac{\lambda^{d-2}}{r_d(\lambda)}\right\},\end{equation}
where $r_d(\lambda)$ is the number of lattice points lying on the circle of radius $\lambda$ 
 centered at $0$ -- see paragraph~\ref{ss:average} for a brief reminder on its asymptotic properties.
 In particular, the upper bound does not necessarily go to $0$ as $\lambda\rightarrow+\infty$. For instance, in dimension $d\geq 5$, 
 one has $r_d(\lambda)\sim \lambda^{d-2}$. Yet, in lower dimension, it is well known that the value of $r_d(\lambda)$ may fluctuate depending on the 
 value of $\lambda$. Let us briefly recall some results in this direction.
 
 In dimension~$2$, if we denote by 
 $\sigma(\Delta)$ the integers which are a sum of two squares (equivalently the set of Laplace eigenvalues), then, according to~\cite[Cor.~3.6]{KuUe14}, one can find, for every 
 $\delta>0$, a density~$1$ subset $S_{\delta}$ of $\sigma(\Delta)$ such that, for every $\lambda^2\in S_{\delta}$,
 $$r_2(\lambda)\geq\left(\log\lambda\right)^{\frac{\log 2}{2}-\delta}.$$
 In particular, for every $\lambda^2\in S_{\delta}$,
 $$\frac{1}{|\{j:\lambda_j= \lambda\}|}\sum_{j:\lambda_j= \lambda}\left|\int_{\IT^2}a|\psi_j|^2dx-\int_{\IT^2}adx\right|^2\leq \frac{c_2'\|a\|_{L^2(\IT^2)}^2}{\left(\log\lambda\right)^{\frac{\log 2}{2}-\delta}}.$$
 More generally, the variance goes to $0$ along any subsequence of eigenvalues such that $r_2(\lambda)\rightarrow+\infty$. 
 On the other hand, for every $q\geq 0$, one has $r_2(3^q)=4$~\cite[p.~15]{Gr85}. Hence, $r_2(\lambda)$ remains bounded along certain subsequences tending 
 to $+\infty$. 
 
 The same discussion occurs in dimensions $d=3,4$. Indeed, 
 one has $r_3(2^q)=6$ and $r_4(2^q\sqrt{2})=24$ for every $q\geq 0$~\cite[p.~38, p.~30]{Gr85} while, along good subsequences of eigenvalues, 
 $\frac{\lambda^{d-2}}{r_d(\lambda)}$ may go to $0$ as $\lambda$ 
 tends to $+\infty$. For instance, in dimension~$4$, we can set, for every $n\geq 1$, $\lambda(n)^2=\prod_{p\leq n:p\in\ml{P}}p$ where $\ml{P}$ is the set 
 of prime numbers. It follows from~\cite[p.~38]{Ap76} and~\cite[p.~30]{Gr85} that $r_4(\lambda(n))=\prod_{p\leq n:p\in\ml{P}}(1+p).$ Hence, one has
 $$\log r_4(\lambda(n))=\sum_{p\leq n:p\in\ml{P}}\log p+\sum_{p\leq n:p\in\ml{P}}\log\left(1+\frac{1}{p}\right)
 =\sum_{p\leq n:p\in\ml{P}}\log p+\sum_{p\leq n:p\in\ml{P}}\frac{1}{p}+\ml{O}(1).$$
Then, from~\cite[p.~90]{Ap76}, one finds
$$r_4(\lambda(n))=\lambda(n)^2\exp\left(\sum_{p\leq n:p\in\ml{P}}\frac{1}{p}+\ml{O}(1)\right)=e^{\ml{O}(1)}\lambda(n)^2\log\ n.$$
In particular, the upper bound in~\eqref{e:variance-individual} tends to $0$ for this sequence of eigenvalues.


\subsubsection*{\textbf{Sharpness in dimension $2$}}

When $d=2$, we have shown that
$$\frac{1}{|\{j:\lambda_j= \lambda\}|}\sum_{j:\lambda_j= \lambda}\left|\int_{\IT^2}a|\psi_j|^2dx-\int_{\IT^2}adx\right|^2\leq\frac{2\|a\|_{L^2}^2}{r_2(\lambda)},$$
and we emphasized that the upper bound may not go to $0$. Here, we exhibit a sequence of eigenvalues $\lambda(q)^2\rightarrow+\infty$ 
such that the variance does not tend to $0$ as $q$ tends to $+\infty$. For this purpose, recall that Iwaniec proved the existence of a sequence of 
integers such that $n_q\rightarrow+\infty$ and such that $\lambda(q)^2:=n_q^2+1$ is the product of at most two primes~\cite{Iw78}. Then, we know 
from~\cite[p.~15]{Gr85} that, for every $q$ large enough, $8\leq r_2(\lambda(q))\leq 16$, hence if the variance sum goes to zero then each individual term must go to zero. However, if we select
$$\varphi_q(x):=\frac{1}{\sqrt{2}}\left(e^{i(n_q x_1+x_2)}+e^{i(n_q x_1-x_2)}\right)=\sqrt{2}e^{in_q x_1}\cos(x_2), $$ then for every $a$ in $\ml{C}^{\infty}(\IT^2)$ and for every $q\geq 1$, one has
$$\int_{\IT^2}a(x)|\varphi_q(x)|^2dx=\int_{\IT^2}2a(x)\cos^2(x_2)dx, $$
which obviously does not converge to $\int_{\IT^2}a(x)dx$.

 \subsection{Proof of Theorem~\ref{t:QE-measure}}\label{ss:proof-fractal-order-2} Fix now a probability measure $\mu$ satisfying the properties 
 of Theorem~\ref{t:QE-measure}. It can be viewed as a distribution on $\IT^d$. Hence, one can find a sequence of 
 smooth functions $(\mu_m)_{m\geq 1}$ such that $\mu_m\rightharpoonup\mu$ in $\ml{D}'(\IT^d)$. We set $a$ to be a smooth 
 function on $\IT^d$ and we apply Theorem~\ref{t:maintheo} to the test function $a\mu_m$: 
$$\frac{1}{N(\lambda)}\sum_{j:\lambda_j\leq \lambda}\left|\int_{\IT^d}a\mu_m|\psi_j|^2dx-\int_{\IT^d}a\mu_mdx\right|^2
\leq \frac{C_d}{\lambda}\sum_{n:1\leq \|n\|\leq 2\lambda}\frac{|\int_{\IT^d}a(x)\mu_m(x) e^{-inx} dx|^2}{\|\hat{n}\|}.$$
 This equality is valid for \emph{any} $m\geq  1$ and \emph{any} $\lambda\geq 1$. By letting $m\rightarrow +\infty$, we find
 \begin{equation}\label{e:L2}\frac{1}{N(\lambda)}\sum_{j:\lambda_j\leq \lambda}\left|\int_{\IT^d}a|\psi_j|^2d\mu-\int_{\IT^d}ad\mu\right|^2
\leq \frac{C_d}{\lambda}\sum_{n:1\leq \|n\|\leq 2\lambda}\frac{|\int_{\IT^d}a(x) e^{-inx} d\mu(x)|^2}{\|\hat{n}\|}.\end{equation}
Recall now that, for every $n$ in $\IZ^d$,
$$|\widehat{\mu}(n)|^2\leq C (1+\|n\|)^{-\alpha}.$$ 
Hence, one has, for every $n$ in $\IZ^d-\{0\}$,
$$\left|\int_{\IT^d}a(x) e^{-inx} d\mu(x)\right|\leq \sum_{p\in\IZ^d}|\hat{a}_p||\widehat{\mu}(n-p)|\leq \frac{C}{(1+\|n\|)^{\frac{\alpha}{2}}}\sum_{p\in\IZ^d}|\hat{a}_p|(1+\|p\|)^{\frac{\alpha}{2}}=\ml{O}_a(\|n\|^{-\frac{\alpha}{2}}),$$
where the constant in the remainder depends on a certain number of derivatives of $a$. From this, one can infer
\begin{equation}\label{e:step-arithmetic}\frac{1}{N(\lambda)}\sum_{j:\lambda_j\leq \lambda}\left|\int_{\IT^d}a|\psi_j|^2d\mu-\int_{\IT^d}ad\mu\right|^2
=\ml{O}_a(\lambda^{-1})\sum_{n:1\leq \|n\|\leq 2\lambda}\frac{1}{\|n\|^{\alpha}\|\hat{n}\|}.\end{equation}
We then argue as in~\cite{LeRu17}:
$$\sum_{1\leq \|n\|\leq\lambda}\frac{1}{\|\hat{n}\|\|n\|^{\alpha}}=
\sum_{1\leq m\leq\lambda}\frac{1}{m^{\alpha}}\sum_{1\leq\|\hat{n}\|\leq\lambda/m}\frac{1}{\|\hat{n}\|^{\alpha+1}}.$$
We distinguish several cases:
\begin{itemize}
 \item $\alpha>d-1$: the sum $\sum_{1\leq\|\hat{n}\|\leq\lambda/m}\frac{1}{\|\hat{n}\|^{\alpha+1}}$ is bounded.
 \item $\alpha=d-1$: the sum $\sum_{1\leq\|\hat{n}\|\leq\lambda/m}\frac{1}{\|\hat{n}\|^{\alpha+1}}$ is of order $\ml{O}(\log(\lambda/m)).$
 \item $\alpha<d-1$: the sum $\sum_{1\leq\|\hat{n}\|\leq\lambda/m}\frac{1}{\|\hat{n}\|^{\alpha+1}}$ is of order $(\lambda/m)^{d-1-\alpha}$.
\end{itemize}
In any case, this gives the upper bound
$$\sum_{1\leq \|n\|\leq\lambda}\frac{1}{\|\hat{n}\|\|n\|^{\alpha}}\lesssim (\log\lambda)^2\lambda^{\max(d-1-\alpha,0)},$$
where $\lesssim$ means that the upper bound involves a constant which is independent of $\lambda$. 
\begin{rema}\label{r:log-term} For $d\geq 3$ and $\alpha<d-1$, the $(\log \lambda)^2$ factor can be removed. For $d\geq 3$ and $\alpha=d-1$, we have a term of order $\log\lambda$ instead of $(\log\lambda)^2$. 
Finally, for $\alpha>d-1$, the $(\log \lambda)^2$ factor can be removed again.
 
\end{rema}

Combined with~\eqref{e:step-arithmetic}, we finally get
$$\frac{1}{N(\lambda)}\sum_{j:\lambda_j\leq \lambda}\left|\int_{\IT^d}a|\psi_j|^2d\mu-\int_{\IT^d}ad\mu\right|^2
=\ml{O}_a\left((\log\lambda)^2\lambda^{\max(d-2-\alpha,-1)}\right),$$
 which concludes the proof of Theorem~\ref{t:QE-measure}.

\begin{rema}\label{r:LR-hypersurface} Regarding the results in Theorem~\ref{t:QE-measure}, we note that this upper bound 
is better than what we would a priori get via the upper bounds from~\cite{LeRu17} on the moments of order $1$. In fact, if we use $1$-moments, then the upper bound we would need to estimate is
$$\sum_{1\leq \|n\|\leq\lambda}\frac{1}{\|\hat{n}\|\|n\|^{\frac{\alpha}{2}}},$$
which would lead to stronger constraints on the value of $\alpha$. By estimating the moments of order $2$, we shall typically impose 
$\alpha>d-2$ while moments of order $1$ would require to fix $\alpha>2(d-2)$ which is a worst constraint when $d\geq 3$. For instance, 
if we we keep in mind that our main application follows from Littman's 
Theorem, we need to allow $\alpha=d-1$, and this would not be authorized by the assumption $\alpha>2(d-2)$ in dimensions $d\geq 3.$
 
\end{rema}

\subsection{Average of eigenfunctions over hypersurfaces}\label{ss:average}

In this paragraph, we make a few well-known remarks on the related question of the average of toral eigenfunctions over an hypersurface $\Sigma$ 
endowed with the induced Riemannian measure $\sigma$. 
More precisely, given a solution $\psi_{\lambda}$ of~\eqref{e:eigenvalue}, we want to estimate
$$\int_{\Sigma}\psi_{\lambda}(x)d\sigma(x).$$
Using Zelditch's Kuznecov trace formula~\cite[Eq.~3.4]{Ze92}, one knows that for a density one subsequence of 
eigenfunctions of a  given orthonormal basis, this quantity decays as 
$\lambda^{\frac{1-d}{2}}R(\lambda)$ for any choice of 
$R(\lambda)\rightarrow +\infty$. In addition, the trace formula of \cite{Ze92} also shows that for all eigenfunctions this quantity is uniformly bounded. 
In the case of the rational torus $\IT^d$, we would like to recall how one can easily say a little bit more provided 
some assumptions are made on the curvature of $\Sigma$. To see this, we write
$$\psi_{\lambda}=\sum_{k:\|k\|=\lambda}\hat{\psi}_{\lambda}(k)\mathbf{e}_k,$$
and we find using the Cauchy-Schwarz inequality
$$\left|\int_{\Sigma}\psi_{\lambda}(x)d\sigma(x)\right|=\left|\sum_{k:\|k\|=\lambda}\hat{\psi}_{\lambda}(k)\hat{\sigma}(-k)\right|\leq\left(
\sum_{k:\|k\|=\lambda}|\hat{\sigma}(k)|^2\right)^{\frac{1}{2}}.$$
Suppose now that, for every $x$ in $\Sigma$, all the principal curvatures at $x$ do not vanish, so that we can apply Littman's Theorem again. This 
yields
$$\left|\int_{\Sigma}\psi_{\lambda}(x)d\sigma(x)\right|\leq\left(
\sum_{k:\|k\|=\lambda}\frac{1}{\|k\|^{d-1}}\right)^{\frac{1}{2}}\leq \frac{1}{\lambda^{\frac{d-1}{2}}}\left|\{k:\|k\|=\lambda\}\right|^{\frac{1}{2}}.$$
Hence, everything boils down to an estimate on the number $r_d(\lambda)$ of lattice points on a circle of radius $\lambda$ which is a standard problem in number 
theory. We briefly recall classical results from the literature. First of all, for $d\geq 5$, it follows 
from~\cite[p.~155]{Gr85} that $r_d(\lambda)$ is of order $\lambda^{d-2}.$ If $d=4$, we can deduce from~\cite[Th.~13.12]{Ap76} 
and~\cite[Th.~4, p.~30]{Gr85} that $r_4(\lambda)$ is bounded by $\lambda^{2+\delta}$ for every $\delta>0$. 
When $d=3$, we have that $r_3(\lambda)$ is bounded by $\lambda^{1+\delta}$ for any 
positive $\delta$ as a consequence of~\cite[Th.~13.12]{Ap76} and~\cite[Eq.~(4.10),p.~55]{Gr85}. For $d=2$, it follows from~\cite[Th.~13.12]{Ap76} 
and~\cite[p.~15,Th.~3]{Gr85} that $r_2(\lambda)$ is of order $\lambda^{\delta}$ for every $\delta>0$. Gathering these bounds, we get:
\begin{prop} Suppose that the assumptions of Theorem~\ref{t:maintheo-hypersurface} are satisfied. 
Then, for any solution $\psi_{\lambda}$ of~\eqref{e:eigenvalue}, one 
has, for every $\delta>0$,
$$\int_{\Sigma}\psi_{\lambda}(x)d\sigma(x)=\ml{O}(\lambda^{-\frac{1}{2}+\delta}),$$
with $\delta$ that can be taken to be $0$ when $d\geq 5$.
\end{prop}

We emphasize that the assumption on $\Sigma$ is sharp in the sense that, if we consider a flat subtorus embedded in $\IT^d$ (e.g. $\IR^{d-1}/(2\pi\IZ)^{d-1}$), then it is easy to 
construct a subsequence of eigenfunctions verifying $\int_{\Sigma}\psi_{\lambda}(x)d\sigma(x)=1$ (take $\psi_k(x)=\exp(i kx_d)$).

\section{The case $d=2$}
In this last section, we focus on the $2$-dimensional case. 
More precisely, we draw a 
couple of simple consequences related to our results from an arithmetic Lemma due to Bourgain and Rudnick~\cite{BoRu11}.

It is known that the number of integers in $\{1,\ldots,N\}$ that can be written as the sum of two squares is of order $N/\sqrt{\log N}$ as $N\rightarrow+\infty$. Among 
these numbers, we know that most of them have the property of having well separated solutions~\cite[Lemma~5]{BoRu11}:
\begin{lemm}\label{l:arithmetic} Let $\delta>0$. Then, for all but $\ml{O}(N^{1-\frac{\delta}{3}})$ integers $E\leq N$, one has
 $$\min_{k\neq l\in\IZ^2:\|k\|^2=\|l\|^2=E}|k-l|> E^{\frac{1-\delta}{2}}.$$
\end{lemm}
In other words, for almost all eigenvalues $\lambda^2\leq \Lambda^2$, the associated lattice points are well separated. In the following, we denote by 
$\sigma(\Delta)$ the set of integers which are a sum of two squares, equivalently the set of Laplace eigenvalues. A nice corollary of this Lemma is the following:
\begin{coro}\label{c:generic-QUE} Let $d=2$. Then, there exists a density $1$ subset $S$ of $\sigma(\Delta)$ such that, 
for every sequence of $L^2$ normalized solutions $\psi_{\lambda}$ of
$$\Delta\psi_{\lambda}=-\lambda^2\psi_{\lambda},$$
one has,
$$\forall a\in\ml{C}^0(\IT^2),\ \lim_{\lambda\rightarrow+\infty,\lambda^2\in S}\int_{\IT^2}a(x)|\psi_{\lambda}(x)|^2dx=\int_{\IT^2}a(x)dx.$$
\end{coro}
In other words, for a generic sequence of eigenvalues, \emph{all} the Laplace eigenfunctions are equidistributed on the configuration space 
as $\lambda\rightarrow +\infty$.
\begin{proof} By a density argument, it is sufficient to prove this for $a=\mathbf{e}_p$ for every $p$ in $\IZ^2$. We write
$$\psi_{\lambda}=\sum_{k:\|k\|=\lambda}\hat{c}_k\mathbf{e}_k.$$
Then, one has
$$\int_{\IT^2}\mathbf{e}_p(x)|\psi_{\lambda}(x)|^2dx=\sum_{k,l:\|k\|=\|l\|=\lambda}\hat{c}_k\overline{\hat{c}_l}\int_{\IT^2}\mathbf{e}_{p+k-l}(x)dx.$$
From Lemma~\ref{l:arithmetic}, we know that the off-diagonal terms cancel for $\lambda$ large enough. Hence, for $\lambda$ large enough (in a way that depends on $p$ in $\IZ^2$), one has
$$\int_{\IT^2}\mathbf{e}_p(x)|\psi_{\lambda}(x)|^2dx=\int_{\IT^2}\mathbf{e}_p(x)dx,$$
which concludes the proof of the corollary.
\end{proof}

\begin{rema} In a recent article, Granville and Wigman \cite{GrWi} obtained improvements on the Bourgain-Rudnick lemma and used it to get some results on \textit{small scale equidistribution} property of toral eigenfunctions in dimension $2$. See \cite{HeRi17, LeRu17, Sa17} for earlier work on this topic.
\end{rema}

Once we have noticed this simple corollary of Bourgain-Rudnick's Lemma, the next question to ask is whether this remains true for equidistribution 
along hypersurfaces. This is in fact the case thanks to Littman's Theorem and to the upper bound from~\eqref{e:bourgain-rudnick}:
\begin{prop}\label{p:generic-QUE-curve} Let $\Sigma\subset\IT^2$ be a smooth closed embedded and oriented curve with nonvanishing curvature. 
Then, there exists a density $1$ subset $S$ of $\sigma(\Delta)$ such that, for every sequence of $L^2$ normalized solutions $\psi_{\lambda}$ of
$$\Delta\psi_{\lambda}=-\lambda^2\psi_{\lambda},$$
one has, as $\lambda\rightarrow+\infty$,
$$\forall a\in\ml{C}^0(\IT^2),\ \lim_{\lambda\rightarrow+\infty,\lambda^2\in S}\int_{\Sigma}a(x)|\psi_{\lambda}(x)|^2d\sigma(x)=\int_{\Sigma}a(x)d\sigma(x)+o(1).$$
\end{prop}
\begin{proof} Using Bourgain-Rudnick's upper bound~\eqref{e:bourgain-rudnick}, we can make use of a density argument and it 
is again sufficient to prove this result when $a=\mathbf{e}_p$ for every fixed $p$ in $\IZ^2$. Again, we write
 $$\psi_{\lambda}=\sum_{k:\|k\|=\lambda}\hat{c}_k\mathbf{e}_k,$$
 and we find that
 $$\int_{\Sigma}\mathbf{e}_p(x)|\psi_{\lambda}(x)|^2d\sigma(x)=\sum_{k,l:\|k\|=\|l\|=\lambda}\hat{c}_k\overline{\hat{c}_l}\int_{\Sigma}\mathbf{e}_{p+k-l}(x)d\sigma(x).$$
Applying Littman's Theorem~\cite{Li63} and Lemma~\ref{l:arithmetic}, one finds 
$$\int_{\Sigma}\mathbf{e}_p(x)|\psi_{\lambda}(x)|^2d\sigma(x)=\int_{\Sigma}\mathbf{e}_pd\sigma+\ml{O}\left(\lambda^{-\frac{1-\delta}{2}}\right)
\sum_{k\neq l:\|k\|=\|l\|=\lambda}|\hat{c}_k||\hat{c}_l|
.$$
For every $\delta>0$, recall from paragraph~\ref{ss:average} that the number of lattice points on the circle $\lambda\mathbb{S}^1$ is of order $\ml{O}(\lambda^{\delta})$. 
Hence, thanks to the Cauchy-Schwarz inequality, the contribution of the nondiagonal terms will be of order $\ml{O}(\lambda^{\delta})$. This implies that
$$\int_{\Sigma}\mathbf{e}_p(x)|\psi_{\lambda}(x)|^2d\sigma(x)=\int_{\Sigma}\mathbf{e}_{p}(x)d\sigma(x)+o(1).$$
\end{proof}

\section*{Acknowledgments}

We are grateful to Suresh Eswarathasan for discussions related to his recent work~\cite{EP17} and to Ze\'ev Rudnick for his comments related to the first version of the manuscript. 
The second author is partially supported by the Agence Nationale de la Recherche through the Labex CEMPI (ANR-11-LABX-0007-01) and 
the ANR project GeRaSic (ANR-13-BS01-0007-01).


\begin{thebibliography}{HHHH}


\bibitem[Ap76]{Ap76} T. M.~Apostol \emph{Introduction to analytic number theory}. Springer-Verlag, New York Inc. (1976).

\bibitem[Bo93]{Bo93} J. Bourgain \emph{Eigenfunctions bounds for the Laplacian on the $n$-torus}. IMRN \textbf{3}
(1993), 61--66.

\bibitem[BoDe14]{BoDe14} J. Bourgain and C. Demeter \emph{The proof of the $l^2$ decoupling conjecture}.  Ann. of Math. (2) \textbf{182} (2015), no. 1, 351--389.


\bibitem[BoRu11]{BoRu11}  J. Bourgain and Z. Rudnick \emph{On the geometry of the nodal lines of eigenfunctions of the two-dimensional torus}. Ann. Henri Poincar\'e \textbf{12} (2011), 1027--1053.

\bibitem[BoRu12]{BoRu12}  J. Bourgain and Z. Rudnick \emph{Restriction of toral eigenfunctions to hypersurfaces and nodal sets}. GAFA $\textbf{22}$ (2012), 878--937.

\bibitem[BGT07]{BGT07} N. Burq, P. G\'erard, and N. Tzvetkov \emph{Restrictions of the Laplace-Beltrami eigenfunctions to submanifolds}. Duke Math. J. $\mathbf{138}$ (2007), 445--486.

\bibitem[CdV85]{CdV85} Y.~Colin de Verdi\`ere \emph{Ergodicit\'e et fonctions propres du Laplacien}. Comm. in Math. Phys. $\mathbf{102}$, 497--502 (1985).

\bibitem[DyZw12]{DyZw12} S.~Dyatlov, M.~Zworski \emph{Quantum ergodicity for restrictions to hypersurfaces}. Nonlinearity \textbf{26} (2013) 35.

\bibitem[DuGu75]{DuGu75} J. J.~Duistermaat and V.~Guillemin \emph{The spectrum of elliptic operators and periodic bicharacteristics}. Inv. Math. $\mathbf{29}$ (1975), 39--79.

\bibitem[EFK95]{EFK95} B.~Eckhardt, S.~Fishman, J.~Keating, O.~Agam, J.~Main, and K.~M\"uller, \emph{Approach to ergodicity in quantum wave functions}, Phys. Rev.  
E 52 (1995), 5893--5903

\bibitem[EP17]{EP17} S.~Eswarathasan, M.~Pramanik, \emph{$L^p$ norms of fractal restrictions of eigenfunctions on surfaces}, in progress (2018). 

\bibitem[FePe86]{FePe86} M.~Feingold and A.~Peres \emph{Distribution of matrix elements of chaotic systems}. Phys. Rev. A \textbf{34} (1986), 591--595.

\bibitem[GrWi17]{GrWi} A. Granville and I. Wigman \emph{Planck-scale mass equidistribution of toral Laplace eigenfunctions}. Comm. Matt. Phys. \textbf{355}(2017), no 2, 767--802. 

\bibitem[Gr85]{Gr85} E.~Grosswald \emph{Representations of integers as sums of squares}. Springer-Verlag, New York, (1985).

\bibitem[Ha17]{Ha17} K.~Hambrook \emph{Explicit Salem sets in $\IR^2$}. Adv. Math. \textbf{311} (2017), 634--648.

\bibitem[HeRi17]{HeRi17} H. Hezari and G.~Rivi\`ere \emph{Quantitative equidistribution properties of toral eigenfunctions}. J. Spectral Theory $\mathbf{7}$ (2017), 471--485.

\bibitem[Hl50]{Hl50} E.~Hlawka \emph{Integrale auf konvexen K\"orpern}, Monatsh. Math. $\mathbf{54}$ (1950), 1--36 

\bibitem[Ho68]{Ho68} L.~H\"ormander, \emph{The spectral function of an elliptic operator}, Acta Math. \textbf{121} (1968), 193--218.

\bibitem[Iw78]{Iw78} H.~Iwaniec \emph{Almost primes represented by quadratic polynomials}, Inv. Math. \textbf{47} (1978), 171--188

\bibitem[KuUe14]{KuUe14} P.~Kurlberg, H.~Uebersch\"ar \emph{Quantum Ergodicity for point scatterers on arithmetic tori}, GAFA \textbf{24} (2014), 1565--1590

\bibitem[KuRu05]{KuRu05} P. Kurlberg and Z. Rudnick \emph{On the distribution of matrix elements for the quantum cat map}, Ann. of Math. \textbf{161} (2005), 489--507.

\bibitem[LeRu17]{LeRu17} S. Lester, Z.~Rudnick \emph{Small scale equidistribution of eigenfunctions on the torus}.  Comm. Math. Phys. $\mathbf{350}$ (2017), 279--300.

\bibitem[Li63]{Li63} W.~Littman \emph{Fourier transforms of surface-carried measures and differentiability of surface averages}. Bull. Amer. Math. Soc.
Vol. \textbf{69} (1963), 766--770.

\bibitem[LuSa95]{LuSa95} W. Luo,  P. Sarnak \emph{Quantum ergodicity of eigenfunctions on $PSL_2(\IZ)\backslash\mathbb{H}^2$}. Publ. Math. IH\'ES \textbf{81} (1995), 207--237.

\bibitem[MaRu12]{MaRu12} J.~Marklof and Z.~Rudnick \emph{Almost all eigenfunctions of a rational polygon are uniformly distributed}. 
J. of Spectral Theory. \textbf{2} (2012), 107--113.

\bibitem[M95]{M95} P.~Matila \emph{Geometry of sets and measures in Euclidean spaces. Fractals and Rectifiability}. 
Cambridge Studies in Advanced Mathematics \textbf{44} (1995).


\bibitem[Sa17]{Sa17} N. T. Sardari \emph{Quadratic forms and semiclassical eigenfunction hypothesis for flat tori}, arXiv:1604.08488, (2017). 

\bibitem[Sch95]{Sc} W. M. Schmidt \emph{Northcott's theorem on heights II. The quadratic case}. Acta Arith. \textbf{70}.4 (1995), 343--375.

\bibitem[Sch06]{Sch06} R. Schubert \emph{Upper bounds on the rate of quantum ergodicity}. Ann. Henri Poincar\'e \textbf{7} (2006), 1085--1098.

\bibitem[Sh74]{Sh74} A.~Shnirelman \emph{Ergodic properties of eigenfunctions}. Usp. Math. Nauk. \textbf{29} (1974), 181--182.


\bibitem[ToZe12]{ToZe12} J.~Toth, S.~Zelditch \emph{Quantum ergodic restriction theorems. I: Interior hypersurfaces in domains with ergodic billiards}. Ann. Henri Poincar\'e \textbf{13} (2012), 599--670.

\bibitem[ToZe13]{ToZe13} J.~Toth, S.~Zelditch \emph{Quantum ergodic restriction theorems, II: manifolds without boundary}. Geom. Funct. Anal. \textbf{23} (2013), 715--775. 

\bibitem[ToZe17]{ToZe17} J.~Toth, S.~Zelditch \emph{Nodal intersections and geometric control}. preprint arXiv:1708.05754 (2017). 

\bibitem[Ze87]{Ze87} S.~Zelditch \emph{Uniform distribution of the eigenfunctions on compact hyperbolic surfaces}. Duke Math. Jour. \textbf{55}, 919--941 (1987).

\bibitem[Ze92]{Ze92} S.~Zelditch \emph{Kuznecov sum formulae and Szeg\"o limit formulae on manifolds}. Comm. Partial Differential Equations \textbf{17} (1992), 221--260.

\bibitem[Ze94]{Ze94} S. Zelditch \emph{On the rate of quantum ergodicity. I. Upper bounds}. Comm. Math. Phys. \textbf{160} (1994), no. 1, 81--92.

\bibitem[Zy74]{Zy74} A. Zygmund  \emph{On Fourier coefficients and transforms of functions of two variables}.  Studia Mathematica, \textbf{50} (1974), no 2, 189--201.


\end{thebibliography}
\end{document}